\documentclass[a4paper,12pt]{amsart}

\usepackage{float}
\usepackage{euscript,eufrak,verbatim}
\usepackage{graphicx}
\usepackage[usenames]{color}
\usepackage[colorlinks,linkcolor=red,anchorcolor=blue,citecolor=blue]{hyperref}
\usepackage{amsmath}
\usepackage{amsthm}
\usepackage[all]{xy}
\usepackage{graphicx} 
\usepackage{amssymb} 
\usepackage{amsfonts}
\usepackage{enumitem}

\usepackage{mathrsfs}
\usepackage{amscd}
\usepackage{pifont}

\usepackage[pagewise]{lineno}

\makeatletter
\newcommand{\svdots}{%
  \vbox{\fontsize{\sf@size}{\sf@size pt}\linespread{0.3}\selectfont
    \kern0.2\baselineskip
    \hbox{.}\hbox{.}\hbox{.}%
    \kern0.1\baselineskip
  }%
}
\makeatother

\theoremstyle{plain}
\newtheorem{main theorem}{Main Theorem}
\newtheorem{theorem}{Theorem}[section]
\newtheorem{lemma}[theorem]{Lemma}

\newtheorem{claim}[theorem]{Claim}

\newtheorem{lemma-definition}[theorem]{Lemma-Definition}
\theoremstyle{definition}
\newtheorem{definition}[theorem]{Definition}
\newtheorem{remark}[theorem]{Remark}

\newtheorem{problem}[theorem]{Problem}

\makeatletter 
\@addtoreset{equation}{section}
\numberwithin{equation}{section}

\newcommand{\supp}{\mathrm{supp}}
\newcommand{\sing}{\mathrm{sing}}
\newcommand{\lip}{\mathrm{Lip}}
\newcommand{\fix}{\mathrm{Fix}}

\newcommand{\osc}{\mathrm{osc}}
\newcommand{\Id}{\mathrm{Id}}
\newcommand{\Lie}{\mathrm{Lie}}
\newcommand{\loc}{\mathrm{loc}}
\newcommand{\dist}{\mathrm{dist}}
\newcommand{\diver}{\mathrm{div}}

\def\R{\mathbb{R}}
\def\N{\mathbb{N}}
\def\Z{\mathbb{Z}}
\def\11{\mathbf{1}}
\def\gg{\mathbf{g}}
\def\LL{\mathcal{L}}
\def\rr{\mathsf{r}}
\def\DD{\mathcal{D}}
\def\CC{\mathcal{C}}
\def\RR{\mathcal{R}}
\def\OO{\mathcal{O}}
\def\UU{\overline{U}}
\def\VV{\overline{V}}
\def\HH{\mathcal{H}}

\title[Lipschitz Bebutov--Kakutani theorem for connected Lie group actions]
{A Lipschitz refinement of the Bebutov--Kakutani dynamical embedding theorem for connected Lie group actions}

 \author{Yonatan Gutman, Qiang Huo, Kangbo Ouyang}

 \address[Yonatan Gutman]
 {Institute of Mathematics, Polish Academy of Sciences, ul. \'{S}niadeckich 8, Warsaw, 00-656, Poland}

\email{gutman@impan.pl}

 \address[Qiang Huo]
 {School of Mathematical Sciences, Beijing University of Posts and Telecommunications, Beijing 100876, China
 \newline \indent Key Laboratory of Mathematics and Information Networks (Beijing University of Posts and Telecommunications), Ministry of Education, Beijing 100876, China}

 \email{qianghuo@ustc.edu.cn}

 \address[Kangbo Ouyang]
 {School of Mathematical Sciences, University of Science and Technology of China, Hefei 230026, China}

 \email{oy19981231@mail.ustc.edu.cn}

\begin{document}

\date{\today}

\maketitle

\begin{abstract}
Let $G$ be a connected Lie group equipped with a compatible right-invariant metric and let $G$ act continuously on a compact metrizable space $X$.
We prove that $X$ admits a $G$-equivariant embedding into 
the right argument shift action on the space of one-Lipschitz maps from $G$ to $[0,1]$ if and only if
the set of fixed points topologically embeds into $[0,1]$.
This generalizes a theorem proved by Gutman, Huo and Tsukamoto for $\R^k$-flows.
\end{abstract}

\section{Introduction}  \label{section: Introduction}

The remarkable Bebutov--Kakutani theorem \cite{Beb40,Kak68} states that a continuous action of $\R$ on a compact metrizable space equivariantly embeds into the shift action on the space of continuous functions on $\R$ if and only if the fixed point set of the $\R$-action embeds into the unit interval. 
The theorem has also been
established for progressively broader classes of acting groups:
connected Lie groups \cite{Chen75a}, connected locally compact groups
\cite{Jaw74,Aus88}, and connected generalized Lie groups in the sense
of Chen \cite{Chen75b}.\footnote{The definition of Chen and Yoh can be
found in \cite[Definition~4]{CY74}. Note that this notion differs from
Gleason's notion of a generalized Lie group; see \cite[Definition 4.1]{Gle51}. The
class defined by Chen and Yoh includes Lie groups, Banach Lie groups,
and all locally compact groups; see \cite[Proposition~13]{CY74}.}
However, the universal space appearing in the Bebutov--Kakutani theorem, namely the space of continuous functions on $\R$, fails to be compact (nor locally compact).
Gutman, Jin and Tsukamoto \cite{GJT19} showed that the space of bounded one-Lipschitz functions on $\R$ is indeed a compact universal space, thereby establishing a Lipschitz refinement of the Bebutov--Kakutani theorem.
This naturally raises the question of whether their result can be extended to acting groups beyond $\R$.
We review recent progress towards this direction in the following paragraph.

In \cite{GH25}, Gutman and Huo established a sufficient criterion for
embedding a continuous $\R^k$-action into the shift $\R^k$-action on
the space of bounded one-Lipschitz functions on $\R^k$. Their
criterion required that the fixed-point set embed into the unit
interval, a necessary condition, and that the stabilizer\footnote{For
the definition of \textit{stabilizer}, see Subsection
\ref{subsection:topological dynamical systems}.} of every non-fixed
point be a discrete subgroup of $\R^k$, an additional condition that
is not necessary.
Subsequently, Gutman, Huo, and Tsukamoto \cite{GHT25} removed the
discreteness assumption on the stabilizers of non-fixed points. They
thereby proved that embeddability of the fixed-point set into the unit
interval is the only obstruction to the equivariant embedding of
continuous $\R^k$--actions. This satisfactory characterization
naturally raises the question of whether an analogous result holds for
actions of non-commutative groups.

In this paper, we are concerned with the situation where the acting group is an arbitrary connected Lie group.
More precisely, let $G$ be a connected Lie group with identity $e$ and let $d$ be a right-invariant compatible metric on $G$. Without loss of generality, we always assume that $\dim G\geq 1$. 
We define $\lip_1(G, I)$ as the space of 
all one-Lipschitz maps (w.r.t. $d$) from $G$ to $I = [0,1]$.
Namely, this is the space of all maps $f\colon G\to I$ satisfying 
$|f(g)-f(h)|\leq d(g,h)$ for all $g, h\in G$.
This space is compact (with respect to the compact-open topology) and metrizable.
The group $G$ acts continuously on $\lip_1(G, I)$ by the right argument shift, i.e. $(gu)(h)=u(hg)$ for $g, h\in G$.

Let $G$ act continuously on a compact metrizable space $X$. 
We call the pair $(X, G)$ a $G$-action. Denote $\fix(X, G)= \{x\in X\mid gx = x \, (\forall g\in G)\}$.
We say that $(X, G)$ equivariantly embeds into 
$\lip_1(G, I)$ if there exists a topological embedding $f\colon X\to \lip_1(G, I)$ such that 
$f(gx)(h)=f(x)(hg)$ for all $x\in X$ and $g,h\in G$.

We are now in position to state the main result of this paper, which generalizes the above-mentioned results.

\begin{theorem}[Main theorem] \label{theorem: main theorem}
Let $(X, G)$ be a $G$-action where $G$ is a connected Lie group.
It equivariantly embeds into $\lip_1(G, I)$ if and only if $\fix(X, G)$ topologically embeds into $I$.
\end{theorem}

 The proof starts with an equivariant embedding into $C(G)$ \footnote{For the definition of $C(G),C_b(G)$ and $C_b^1(G)$, see Subsection \ref{Subsec:smooth functions}.}. The main
new ingredient is an injective $G$-equivariant operator $Q:C_b(G)\to
C_b^1(G)$, constructed from the heat flow on $G$. This operator
smooths functions while preserving the information needed for
injectivity and providing quantitative control of their derivatives.
Applying $Q$ to a suitably rescaled embedding therefore preserves
injectivity and equivariance while producing functions with the
required Lipschitz bound, giving the desired equivariant embedding
into $\lip_1(G,I)$.

\begin{remark}
     The connectedness assumption cannot in general be omitted. Indeed, let
$G=\R\times\Z/2\Z$ and $X=[0,1]^3\times\{0,1\}$. Let $\R$ act trivially
on $X$, and let the nontrivial element of $\Z/2\Z$ interchange the two
copies of $[0,1]^3$. Then $\fix(X,G)=\emptyset$. Suppose that
$\Phi\colon X\to\lip_1(G,[0,1])$ is equivariant. Since
$\R$ acts trivially on $X$, equivariance implies that every function
in $\Phi(X)$ is invariant under right translations by the subgroup
$\R\times\{0\}$. Such a function is constant on each of the two
components of $G$ and is therefore completely determined by its values
on these components. Consequently, $\Phi(X)$ is contained in a
subspace naturally identified with a subset of $[0,1]^2$. Since $\dim
X=3$, the map $\Phi$ cannot be an embedding.
\end{remark}

Theorem \ref{theorem: main theorem} also provides a complete answer to \cite[Problem 1.8]{GHT25}.
It is worth mentioning that, when the acting group is a discrete group - for instance, $\Z^k$ or discrete amenable group - and the phase space is infinite-dimensional, the equivariant embedding problem exhibits strikingly different nature and is usually far more challenging. In this case, mean topological dimension, a new topological invariant introduced by Gromov \cite{Gro99}, turns out to be a new obstruction. 
Subsequent works have been developed along this line, see e.g. \cite{LW00,Lin99,LT14,GT20,GQT19,Nar24,DL25,Mey26}.

  We end this section with the following open question.
 \begin{problem}
 Let $G$ be a connected, locally compact, second-countable group, and let $d$ be a proper\footnote{The metric $d$ is called proper if every closed ball in terms of $\{h\in G:d(g,h)\leq r\}$, where $g\in G,r>0$, is compact.} compatible right-invariant metric on $G$. Is it true that a $G$-action $(X,G)$ equivariantly embeds into $\lip_1(G, I)$ if and only if $\fix(X, G)$ topologically embeds into $I$?
 \end{problem}

\section{Preliminaries}\label{Section:preliminaries}

Throughout this section, let $G$ be a smooth Lie group and $U\subset G$ be open.
 Let $\gg$ be a right-invariant Riemannian metric tensor on $G$, which means for every $a\in G$, the right translation defined by $\rr_a: (G,\gg)\to (G,\gg), g\mapsto ga$ is an isometry. Let $\mu$ denote the Riemannian volume measure induced by $\gg$. Then $\mu$ is right-invariant, i.e.  $\mu(Eg)=\mu(E)$ for every Borel set $E\subset G$ and $g\in G$. Let $\rho$ be the geodesic distance induced by $\gg$. We refer to $\rho$ as the  Riemannian distance induced by $\gg$.  Set $n=\dim G$.
For $1\leq p<\infty$, denote by $L^p(G,\mu)$ the space of $p$-integrable functions on $G$ w.r.t. $\mu$, while $L^{\infty}(G,\mu)$ consists of all  measurable functions that are essentially bounded. $C_c(G)$ denotes the space of continuous functions with compact support in $G$. We write $V\Subset U$ if the closure of $V$ is contained in $U$.

\subsection{Spaces of smooth functions}\label{Subsec:smooth functions}

\begin{definition}[Smooth function]
A function $f:U\to\R$ is called \textbf{smooth}, or of class
$C^\infty$ (respectively continuously differentiable, or of class $C^1$), if for every smooth coordinate chart $\kappa:V\to \kappa(V)\subset\mathbb R^n,  V\Subset U$,
the function
\[
f\circ\kappa^{-1}:\kappa(V)\to\R
\]
is a $C^\infty$ (respectively $C^1$) function in the usual Euclidean sense.
\end{definition}

Denote by $C(G)$ (respectively $C^1(G)$) the space of all continuous (respectively continuously differentiable) functions from $G$ to $\R$. Denote by $C_b(G)$ the space of real-valued bounded continuous functions on $G$.
For $u\in C_b(G)$, set
\begin{equation*}
    \osc(u)=\sup_{G}u-\inf_{G}u \text{  and  } \lip_\rho(u)=\sup_{g\neq h}\frac{|u(g)-u(h)|}{\rho(g,h)}.
\end{equation*}

For $u\in C^1(G)$, the \textbf{Riemannian gradient} $\nabla u$ is the
continuous vector field on $G$, which in local coordinates $x^1,\ldots,x^n$ has the form
\[\nabla u=\sum_{i,j}g^{ij}\frac{\partial u}{\partial x^i}\frac{\partial}{\partial x^j},\]
where $(g^{ij})$ is the inverse of the coordinate matrix $(g_{ij})$ of $\gg$.

Define
\begin{equation*}
    C_b^1(G)=\{u\in C^1(G): \|u\|_{\infty}+\|\nabla u\|_{\infty}<\infty\},
\end{equation*}
 where $\|\nabla u\|_\infty=\sup_{g\in G}|\nabla u(g)|_{\gg}$,
and $|\cdot|_{\gg}$ is the norm induced by the Riemannian metric tensor.

\subsection{Distributions}
\begin{definition}
A \textbf{distribution} on $U$ is a continuous linear functional $u:C_c^\infty(U)\to \R$
where $C_c^\infty(U)$ denotes the space of smooth compactly supported
functions on $U$, endowed with its usual test-function topology.
 More precisely, a sequence $(\varphi_j)_{j\in\N}$ converges to $0$ in $C_c^\infty(U)$ if and only if there exist a compact
set $K\subset U$, finitely many smooth coordinate charts $
\kappa_\ell:V_\ell\longrightarrow\kappa_\ell(V_\ell)\subset\mathbb R^n$,
$\ell=1,\ldots,N,
$
and compact sets $K_\ell\subset V_\ell\subset U$ such that
$
K\subset\bigcup_{\ell=1}^N K_\ell,
$ and $\operatorname{supp}\varphi_j\subset K$
for every $j$,
and, for every $\ell=1,\ldots,N$ and every multi-index
$\alpha\in\mathbb N_0^n$,
\[
\sup_{y\in\kappa_\ell(K_\ell)}
\left|
\partial^\alpha
\bigl(\varphi_j\circ\kappa_\ell^{-1}\bigr)(y)
\right|
\longrightarrow 0
\qquad\text{as }j\to\infty.
\]
We write $\langle u,\varphi\rangle, \varphi\in C_c^\infty(U)$
for the value of $u$ on $\varphi$.
\end{definition}
Every function $f\in L^1_{\mathrm{loc}}(U,\mu)$ defines a distribution by
\[
\langle f,\varphi\rangle
=
\int_U f(x)\varphi(x)\,d\mu(x),
\qquad
\varphi\in C_c^\infty(U).
\]
In this way, locally integrable functions are naturally identified with
distributions.

Similarly, a \textbf{distributional vector field} on $U$ is a
continuous linear functional on the space $C_c^\infty(U,TU)$
of smooth compactly supported vector fields, endowed with the
analogous test-function topology, defined by applying the
preceding convergence conditions to each coordinate component.

\begin{definition}[Smooth distribution]
Let $u$ be a distribution on $U$, and let $V\subset U$ be open.
We say that $u$ is \textbf{smooth} on $V$ if there exists
$f\in C^\infty(V)$ such that
\[
\langle u,\varphi\rangle
=
\int_V f(x)\varphi(x)\,d\mu(x)
\]
for every $\varphi\in C_c^\infty(V)$.
Equivalently, the restriction $u|_V$ is the distribution induced by a
smooth function on $V$.
\end{definition}

\subsection{Laplace--Beltrami operator}
Let $X$ be a distributional vector field  on $U$ and $u$ be a distribution on $U$. 
The \textbf{divergence} of $X$, denoted by $\diver X$, is defined by
\[\langle\diver X,\varphi\rangle=-\langle X,\nabla\varphi\rangle,\qquad \forall\varphi\in C_c^{\infty}(U).\]
 The \textbf{distributional
gradient} of $u$ is the distributional vector field $\nabla u$
defined by
\[
\langle \nabla u,Y\rangle
:=-\langle u,\operatorname{div}Y\rangle,
\qquad Y\in C_c^\infty(U,TU),
\]

The \textbf{Laplace--Beltrami operator} \footnote{In this paper, we use the convention that the Laplace–Beltrami operator is positive semidefinite. The opposite sign convention is also used in the literature; see e.g. \cite[Subsection 2.1]{Gri06}.} is defined by
\[\Delta u:=-\diver(\nabla u),\]
such that for all $\varphi\in C_c^{\infty}(U)$ it holds $\langle \Delta u, \varphi\rangle= \langle \nabla u,\nabla\varphi\rangle$.

\subsection{Sobolev spaces}
Let $k\in\{1,2\}$ and $1<p<\infty$.
For an open set $\Omega\subset\mathbb R^n$, we use the
standard Sobolev space $W^{k,p}(\Omega)$, consisting of
functions whose weak partial derivatives of order at most
$k$, including the function itself, belong to $L^p(\Omega)$.
We write $f\in W_{\mathrm{loc}}^{k,p}(\Omega)$ if
$f|_{\Omega'}\in W^{k,p}(\Omega')$ for every open set
$\Omega'\Subset\Omega$. The space $W_0^{1,p}(\Omega)$ is the closure of
$C_c^\infty(\Omega)$ in $W^{1,p}(\Omega)$.

\subsection{Singular support}
\begin{definition}[Singular support]\label{def:singular_support}
Let $u$ be a distribution on $U$. The \textbf{singular support} of $u$,
denoted by $\sing\supp u$,
is the complement in $U$ of the largest open subset on which $u$ is smooth.
Equivalently,
\[
\operatorname{sing\,supp}u
=
U\setminus
\bigcup
\left\{
V\subset U:
V\text{ is open and }u|_V\text{ is smooth}
\right\}.
\]
\end{definition}
Thus $\operatorname{sing\,supp}u$ is the smallest closed subset of $U$
such that $u$ is smooth on its complement.

\subsection{Essentially self-adjoint operators}

Let $\HH$ be a Hilbert space with inner product
$\langle\cdot,\cdot\rangle$. Let $\DD(T)$ be a dense linear
subspace of $\HH$, and let $T:\DD(T)\to\HH$ be a linear operator. 
A linear subspace $\DD$ of $\DD(T)$ is called a \textbf{core} for $T$ if for every $x\in\DD(T)$, there exists a sequence $(x_n)_{n\in\N}\subset\DD$ such that $x_n\to x$ and $Tx_n\to Tx$ in $\HH$ as $n\to\infty$. 
Set
\[\DD(T^*)=\{y\in \HH:\exists u\in \HH\text{ such that } \langle Tx,y\rangle=\langle x,u\rangle, \forall x\in\DD(T)\}.\]
Since $\DD(T)$ is dense in $\HH$, the element $u\in \HH$ (when exists) satisfying $\langle Tx,y\rangle=\langle x,u\rangle$ for all $x\in\DD(T)$ is uniquely determined by $y$. Set $T^*y=u$.
The linear operator $T^*$ is called the \textbf{adjoint operator} of $T$.

 Let $T$ be a densely defined linear operator on $\HH$. Then $T$ is called \textbf{symmetric} if $\langle Tu,v\rangle=\langle u,Tv\rangle$ for all $u,v\in\DD(T)$. We further say $T$ is \textbf{self-adjoint} if $T=T^*$ and $T$ is \textbf{essentially self-adjoint} if its closure $\overline{T}$ is self-adjoint\footnote{Note that essential self-adjointness is a strictly weaker property than self-adjointness.}. Equivalently, $T$ has a unique self-adjoint extension, namely its closure.

\subsection{Heat semigroup and heat kernel}
Assume the Riemannian manifold $(G,\gg)$ is geodesically complete. 
Let $\Delta\vert_{C_c^{\infty}(G)}$ be the Laplace--Beltrami  operator defined on $C_c^{\infty}(G)$. By \cite[Theorem 2.2]{Gri06}, $\Delta\vert_{C_c^{\infty}(G)}$ is essentially self-adjoint. Denote by $\Delta$ the self-adjoint extension of $\Delta\vert_{C_c^{\infty}(G)}$, whose domain $\DD(\Delta)$ is dense in $L^2(G,\mu)$.
By the spectral theorem \cite[Theorem 5.2]{Sch12}, it generates a strongly continuous contraction semigroup on $L^2(G,\mu)$, denoted by $\{e^{-t\Delta}\}_{t\geq 0}$ and named the \textbf{heat semigroup}.

Equivalently, for every $f\in L^2(G,\mu)$, $u(t,h):=(e^{-t\Delta}f)(h)$ is the unique solution satisfying the \textbf{heat equation} 
\begin{equation*}
    \partial_tu(t,h)+\Delta u(t,h)=0
\end{equation*}
and the initial condition, i.e. $u(t,\cdot)$ converges to $f$ in $L^2(G,\mu)$ as $t\to 0$.
Moreover, by \cite[Theorem 3.3]{Gri06}, there exists a smooth function $p_t(g,h)$ on $\R_+\times G\times G$, called the \textbf{heat kernel}, such that
\[(e^{-t\Delta}u)(g)=\int_G p_t(g,h)u(h)d\mu(h)\]
for all $u\in L^2(G,\mu), t>0$ and $g\in G$.

\subsection{Lie theory}
Given a connected Lie group $G$ with Lie algebra $\mathfrak g$, denote by $\exp: \mathfrak g\to G$ the exponential map of $G$. A connected Lie group has no proper open subgroups. The following result, known as closed subgroup theorem, can be found in Theorem 20.10 and Corollary 20.11 of \cite{Lee03}.
\begin{theorem}\label{theorem: closed subgroup theorem}(Closed subgroup theorem)
    Suppose $G$ is a Lie group and $H$ is a subgroup of $G$. Then one has the following equivalence:
    $H$ is closed in $G$ $\iff H$ is an embedded Lie subgroup of $G\iff$ $H$ is an embedded submanifold of $G$.
\end{theorem}
By Theorem \ref{theorem: closed subgroup theorem}, a Lie group $G$ is connected if, for every $g,h\in G$ there exists a continuous map $\gamma:[0,1]\to G$ with $\gamma(0)=g$ and $\gamma(1)=h$.

A metric $\rho$ on $G$ is called \textbf{right-invariant} if for all $g,h,a\in G$ it holds $\rho(ga,ha)=\rho(g,h)$.
Every connected Lie group admits a right-invariant Riemannian metric tensor $\gg$. The Riemannian distance induced by $\gg$ is then right-invariant in the above sense.

\subsection{Geodesic completeness}
Suppose $(M,\gg)$ is a Riemannian manifold, where $M$ is a smooth manifold and $\gg$ is a Riemannian metric tensor. $(M,\gg)$ is said to be \textbf{geodesically complete} if every maximal geodesic is defined for all time. Recall a geodesic $\gamma:I\to M$ is maximal if it cannot be extended to a geodesic on an interval larger than $I$. $(M,\gg)$ is said to be \textbf{metrically complete} if it is complete as a metric space with respect to the Riemannian metric, i.e. every Cauchy sequence in $M$ converges. 
\begin{theorem}\label{theorem:Hopf-Rinow}(Hopf--Rinow, see \cite[Theorem 6.19]{Lee03})
    A connected Riemannian manifold is metrically complete if and only if it is geodesically complete.
\end{theorem}
Every Lie group equipped with a right-invariant Riemannian metric is metrically complete with respect to the induced Riemannian distance. Hence, by the Hopf–Rinow theorem, it is geodesically complete.

\subsection{Topological dynamical systems}\label{subsection:topological dynamical systems}
A \textbf{topological dynamical system} (TDS) is a pair $(X,G)$ equipped with a continuous action $G\times X\to X$, where $X$ is a compact metrizable space and $G$ is a connected Lie group. One refers to $(X,G)$ as a $G$-\textbf{action}. $\R$-actions are sometimes referred to as \textbf{continuous flows}.
The set of \textbf{fixed points} of $(X,G)$ is denoted by $\fix(X, G)= \{x\in X\mid gx = x \, (\forall g\in G)\}$.
Define the \textbf{stabilizer of} $x\in X$ by $G_x=\{g\in G: gx=x\}$, which is a closed subgroup of $G$.
A morphism between two $G$-actions $(X,G)$ and $(Y,G)$ is given by a continuous map $\varphi: X\to Y$ which is $G$-\textbf{equivariant}, i.e. $\varphi(gx)=g\varphi(x)$ for all $g\in G$ and $x\in X$. If $\varphi$ is an injective morphism, it is called a $G$-\textbf{equivariant embedding}. 

\subsection{Spaces of Lipschitz functions}\label{subsection:spaces}
Let $G$ be a connected Lie group and $\rho$ be a right-invariant Riemannian distance on $G$. Fix a right-invariant compatible metric $d$ on $G$.

 A function $u:G\to \R$ is called \textbf{one-Lipschitz} (w.r.t. $d$) if $|u(g)-u(h)|\leq d(g,h)$ for all $g,h\in G$.
Let $I$ be the unit interval and $\lip_1(G,I)$ be the space of functions $u: G\to I$ which are one-Lipschitz (w.r.t. $d$).
Equip $\lip_1(G,I)$ with the compact-open topology. Then $\lip_1(G,I)$ is compact.

For $a\in G$, we define the \textbf{right argument shift} $R_a$ on  $C(G), C_b(G)$ and $\lip_1(G,I)$ by
\[(R_au)(g)=u(ga).\]

Denote by $C(X)$ the space of continuous functions from $X$ to $\R$. For $f\in C(X)$, define $J_f: X\to C(G)$  by $J_f(x)(g)=f(gx)$.

We use the compatible metric $d$ to define Lipschitz continuity, while the Riemannian distance $\rho$ serves as an auxiliary metric throughout the proof. The following lemma allows us to deal with these two right-invariant metrics flexibly.

\begin{lemma}\label{lemma: comparison between compatible metrics}
Let $G$ be a connected Lie group. Let $\rho$ be a right-invariant Riemannian distance on $G$ and $d$ be a compatible right-invariant metric on $G$.  There exist
$c,r,\eta>0$ such that for all $g,h\in G$
\begin{align}
 \rho(g,h)<r&\quad\Longrightarrow\quad d(g,h)\ge c\rho(g,h),
 \label{eq:local-metric-comparison}\\
 \rho(g,h)\ge r&\quad\Longrightarrow\quad d(g,h)\ge\eta.
 \label{eq:far-metric-comparison}
\end{align}
\end{lemma}
\begin{proof}
   Let $\mathfrak g=T_eG$ be the Lie algebra of $G$.  Equip
$\mathfrak g$ with the norm induced by the Riemannian metric at $e$.
Choose $a>0$ small enough such that the $\exp:\mathfrak g\to G$ is injective on $\{Y\in\mathfrak{g}:\|Y\|\leq 2a\}$ and there exists $C_0\geq1$ satisfying $ \rho(e,\exp Y)\leq C_0\|Y\|$ if $\|Y\|\leq 2a$.
Then
\[
 \delta_0:=\min_{a\le\|Y\|\le2a}d(e,\exp Y)>0.
\]

If $0<\|Y\|<a$, set $k=\lceil a/\|Y\|\rceil$.  Then
$a\le k\|Y\|<2a$ and thus
\begin{align*}
 \delta_0
 &\le d(e,\exp(kY))
   =d(e,(\exp Y)^k)\\
 &\le\sum_{j=0}^{k-1}d((\exp Y)^j,(\exp Y)^{j+1})\\
 &=k d(e,\exp Y)
  \le \frac{2a}{\|Y\|}d(e,\exp Y).
\end{align*}
The last equality uses the fact $d$ is right-invariant.
Since $ \rho(e,\exp Y)\leq C_0\|Y\|$,
\[
 d(e,\exp Y)\ge
 \frac{\delta_0}{2a}\|Y\|
 \ge \frac{\delta_0}{2aC_0}\rho(e,\exp Y).
\]
Choose $r>0$ small enough such that
$\rho(g,e)<r$ yields $g=\exp Y$ for some $Y\in\mathfrak{g}$ with $\|Y\|\leq a$. Denote $c=\delta_0/(2aC_0)$. Then
$d(e,g)\geq c\rho(e,g)$ if $\rho(e,g)<r$. Conclude \eqref{eq:local-metric-comparison} holds. Indeed, since both $\rho$ and $d$ are right-invariant, if $\rho(g,h)=\rho(e,hg^{-1})<r$, we have $d(g,h)=d(e,hg^{-1})\geq c\rho(e,hg^{-1})=c\rho(g,h)$.

Since $d$ and $\rho$ are compatible, 
there exists $\eta>0$ such that $d(g,h)<\eta$ yields $\rho(g,h)<r$.  Consequently, \eqref{eq:far-metric-comparison} holds.
\end{proof}

\subsection{Miscellaneous}
\begin{lemma}(Scheff\'{e}, \cite{Sche47})\label{lemma:Scheffe}
    Let $f_n,f\in L^1(G,\mu)$ be non-negative functions. If $f_n\to f$ almost everywhere and
    \[\int_G f_nd\mu\to\int_G f d\mu\]
    as $n\to\infty$, then $f_n\to f$ in $L^1(G,\mu)$.
\end{lemma}

\section{The classical Bebutov-Kakutani theorem for connected Lie group actions}  
\label{Section:classical Bebutov-Kakutani}

\begin{theorem}\label{theorem: embedding into C(G,R)}
    Let $(X, G)$ be a $G$-action where $G$ is a connected Lie group and $\fix(X,G)$ embeds into $\R$. Then there exists $f\in C(X)$ such that $J_f:X\to C(G)$ is a $G$-equivariant embedding.
\end{theorem}

\begin{remark}
    Chen \cite[Theorem 2]{Chen75a} extended the classical Bebutov--Kakutani theorem \cite{Kak68} for $\R$-flows to the context of connected Lie group actions, showing that a $G$-action on $X$ admits a $G$-equivariant embedding into the right argument shift on $C(G)$ if and only if $\fix(X,G)$ topologically embeds into $\R$. This result directly yields Theorem \ref{theorem: embedding into C(G,R)}. Indeed, assuming that $\fix(X,G)$ embeds into $\R$, Chen's theorem guarantees the existence of a $G$-equivariant embedding $\Phi: X\to C(G)$. Define $f\in C(X)$ by $f(x)=\Phi(x)(e), x\in X$. Then for all $x\in X$ and $g\in G$ we have $J_f(x)(g)=f(gx)=\Phi(gx)(e)=\Phi(x)(g)$ and thus $J_f=\Phi$. The original proof of the Bebutov--Kakutani theorem relies on two technical lemmas \cite[Lemmas 1 and 2]{Kak68}. Chen stated that these lemmas remain valid for connected Lie group actions,  thereby establishing the desired generalization. Below we give a self-contained proof of Theorem \ref{theorem: embedding into C(G,R)} for the sake of completeness.
\end{remark}

\subsection{Auxiliary lemmas}
The proof of the following lemma can be found in \cite[Lemma 3.2]{GJT19}, see also \cite[Chapter 13, Corollary 6]{Aus88}.
\begin{lemma}\label{lemma:flow box}
    Let $T:\R\times X\to X$ be a continuous flow on a compact metrizable space $X$. If $x_0\in X\setminus\fix(X,T)$, then there exist $a>0$, a closed set $S\subset X$ containing $x_0$ and an open neighborhood $W$ of $x_0$ such that
    the map
    \begin{equation*}
       [-a,a]\times S\to X,\qquad (t,x)\mapsto T_t x
    \end{equation*}
    is injective and its image contains $W$. Moreover, after shrinking $W$, there are continuous maps
    \begin{equation*}
        r:W\to(-a,a),\qquad s:W\to S
    \end{equation*}
    satisfying $x=T_{r(x)}s(x)$ for every $x\in W$ and $r(x_0)=0,s(x_0)=x_0$. We call $(a,S)$ a \textbf{local section} around $x_0$ and $W$ a \textbf{flow box} for $S$.
\end{lemma}

The following lemma is tailored to our proof. The underlying idea is that Theorem \ref{theorem: embedding into C(G,R)} can be reduced to the one-dimensional case, relying on the fact that a connected Lie group contains an abundance of one-parameter subgroups.

\begin{lemma}\label{lemma:one-dimensional flow}(cf. \cite[Lemma 1]{Chen75b})
    Let $T:\R\times X\to X$ be a continuous flow on a compact metrizable space $X$. Denote
    \[\Omega:=X\times X\setminus\left((\fix(X,T)\times\fix(X,T))\cup\Delta_X\right),\]
    where $\Delta_X=\{(x,x):x\in X\}$ is the diagonal of $X\times X$.
    Let $A\subset \fix(X,T)$ be closed and $h: A\to\R$ be continuous. Define $\CC_{A,h}:=\{f\in C(X): f\vert_{A}=h\}$ and
    \begin{equation*}
        \RR:=\{f\in \CC_{A,h}: \text{ for any } (x,y)\in \Omega, \text{ there exists }t\in\R\text{ with }f(T_tx)\neq f(T_ty)\}.
    \end{equation*}
    Then $\RR$ is a dense $G_\delta$ subset of $\CC_{A,h}$.

\end{lemma}

\begin{proof}
    Equip $\CC_{A,h}$ with the uniform norm. Since
$A$ is closed in the compact metric space $X$, the Tietze extension
theorem shows that $\CC_{A,h}$ is nonempty. Moreover, $\CC_{A,h}$ is a
closed  subspace of the Banach space $C(X)$, hence is complete and
therefore a Baire space.
    
    Fix $(x_0,y_0)\in\Omega$. Without loss of generality, we assume that $x_0\notin \fix(X,T)$. Let $(a,S)$ be the local section around $x_0$ and $W$ be the flow box for $S$ given by Lemma \ref{lemma:flow box}, with maps $r:W\to(-a,a)$ and $s:W\to S$. Consider the
    continuous map
   \[
        \Theta:W\times W\times X\to X\times X,
        \qquad
        \Theta(x_1,x_2,y)=\bigl(s(x_1),T_{-r(x_2)}y\bigr).
   \]
   Then $\Theta(x_0,x_0,y_0)=(x_0,y_0)\notin\Delta_X$. Choose open neighborhoods $U$ and $V$ of $x_0$ and $y_0$ respectively such that $\overline{U}\subset W\setminus\fix(X,T)$ and $\Theta(\overline{U}\times\overline{U}\times \overline{V})\subset X\times X\setminus\Delta_X$. Then $(\overline{U}\times \overline{V})\cap (\fix(X,T)\times \fix(X,T))=\emptyset$ and $(\overline{U}\times \overline{V})\cap\Delta_X=\emptyset$. Conclude $\overline{U}\times \overline{V}\subset\Omega$.

   For $x\in\overline{U}$, write $x=T_{r(x)}s(x)$ and define 
   \[S_0=s(\overline{U}), B=\{T_{-r(x)}y: x\in \overline{U},y\in \overline{V}\}.\]
   Then $S_0$ and $\fix(X,T)$ are disjoint. Furthermore, $S_0$ and $B$ are disjoint. Indeed, if $s(x_1)=T_{-r(x_2)}y$ for some $x_1,x_2\in\overline{U}$ and $y\in\overline{V}$, we have $\Theta(x_1,x_2,y)\in\Delta_X$, which leads to a contradiction. Thus $\eta:=\dist(S_0,B)=\inf_{x\in S_0,y\in B}d(x,y)>0$, where $d$ is a compatible metric on $X$.

   Choose $0<\delta<a$ such that for all $x\in X$, if $|t|\leq\delta$, then $d(T_tx,x)<\eta/3$. Set $J=[-\delta,\delta]$. Then the tubes
   \[C_1=T_J(S_0),\qquad C_2=T_J(B)\]
   are disjoint. Indeed, if $T_ts=T_ub$ for some $s\in S_0$, $b\in B$ and $t,u\in J$,
   $d(s,b)\leq d(s,T_ts)+d(T_ub,b)<2\eta/3$. This contradicts $d(s,b)\geq \eta$. Note the map $J\times S_0\to C_1,(t,s)\mapsto T_ts$ is injective as $J\subset(-a,a)$ and $S_0\subset S$, and therefore a homeomorphism onto $C_1$. Moreover, $C_1\cap A=\emptyset$ since $S_0\cap\fix(X,T)=\emptyset$. Otherwise $T_ts\in A\subset\fix(X,T)$ for some $t\in J$ and $s\in S_0$, which yields $s\in \fix(X,T)\cap S_0$.

   Denote 
   \begin{equation*}
       \OO(\overline{U},\overline{V})=\{f\in \CC_{A,h}: \text{ for any } (x,y)\in\overline{U}\times \overline{V}, \text{ there exists }t\in\R\text{ with }f(T_tx)\neq f(T_ty)\}.
   \end{equation*}
   Obviously, $\OO(\UU,\VV)$ is open in $\CC_{A,h}$. We now prove that $\OO(\UU,\VV)$ is dense in $\CC_{A,h}$. Fix $f\in \CC_{A,h}$ and $\varepsilon>0$. For $s\in S_0$ and $b\in B$, define $\alpha_s,\beta_b\in C(J)$ by
   \[\alpha_s(t)=f(T_ts),\qquad \beta_b(t)=f(T_tb).\]
   Then the map $S_0\times B\to C(J), (s,b)\mapsto\beta_b-\alpha_s$ is continuous. Therefore
   \[D_f:=\{\beta_b-\alpha_s:(s,b)\in S_0\times B\}\]
   is a compact subset of $C(J)$. Choose $v\in C(J)$ such that $\|v\|_{\infty}<\varepsilon$ and $v\notin D_f$. Indeed, if $0\notin D_f$, we may choose $v=0$. Otherwise, consider the open ball $\{w\in C(J):\|w\|_{\infty}<\varepsilon\}$. Since $C(J)$ is infinite-dimensional and $D_f$ is compact, by Riesz's lemma the open ball is not contained in $D_f$ and thereby we may find $v$ as desired. 

   Let $E:=C_1\cup C_2\cup A$. Define continuous $q_0:E\to\R$ by
   \[q_0(T_ts)=v(t), \text{ if } t\in J,s\in S_0,\qquad q_0\vert_{C_2\cup A}=0.\]
   $q_0$ is well-defined because $C_1=T_J(S_0)$ is disjoint from $C_2\cup A$.
   By the Tietze extension theorem, we obtain $q\in C(X)$ such that $q\vert_E=q_0$ and $\|q\|_{\infty}\leq\|v\|_{\infty}\leq\varepsilon$. Then $f_1:=f+q\in \CC_{A,h}$ because $q\vert_A=q_0\vert_A=0$.

   For $x\in\UU$ and $y\in\VV$, write $x=T_{r(x)}s(x)$ and set $b=T_{-r(x)}y\in B$. For every $t\in J$, $T_{t-r(x)}x=T_ts(x)$ and $T_{t-r(x)}y=T_tb$. Since $T_ts(x)\in C_1$ and $T_tb\in C_2$, we have
   \[f_1(T_{t-r(x)}x)-f_1(T_{t-r(x)}y)=f_1(T_ts(x))-f_1(T_tb)=\alpha_{s(x)}(t)+v(t)-\beta_b(t).\]
   If $f_1\notin\OO(\UU,\VV)$, we obtain that $v=\beta_b-\alpha_{s(x)}\in D_f$, which leads to a contradiction. Conclude $f_1\in\OO(\UU,\VV)$. Thus $\OO(\UU,\VV)$ is dense in $\CC_{A,h}$ since $f\in \CC_{A,h}$ and $\varepsilon>0$ are arbitrary.

   Up to now, we have proven that for every pair $(x_0,y_0)\in\Omega$, there exist open neighborhoods $U$ and $V$ of $x_0$ and $y_0$ such that $\UU\times\VV\subset\Omega$ and $\OO(\UU,\VV)$ is open and dense in $\CC_{A,h}$. Note $X\times X$ is a hereditarily Lindel\"{o}f space.
   Then there exist countable families of open subsets $\{U_n\}_{n\in\N}$ and $\{V_n\}_{n\in\N}$ of $X$ satisfying $\Omega=\bigcup_{n=1}^{\infty} U_n\times V_n$. Moreover, for each $n\in\N$, $\UU_n\times\VV_n\subset\Omega$ and $\OO(\UU_n,\VV_n)$ is dense in $\CC_{A,h}$.
   By the Baire category theorem, $\bigcap_{n=1}^{\infty}\OO(\UU_n,\VV_n)$ is a dense $G_{\delta}$ subset of $\CC_{A,h}$.
   
   We now prove that $\RR=\bigcap_{n=1}^{\infty}\OO(\UU_n,\VV_n)$.
   Fix $(x,y)\in\Omega$ and $f\in \bigcap_{n=1}^{\infty}\OO(\UU_n,\VV_n)$. Then $(x,y)\in\UU_n\times\VV_n$ for some $n\in\N$. Hence there exists $t\in\R$ with $f(T_tx)\neq f(T_ty)$ since $f\in\OO(\UU_n,\VV_n)$. It follows $f\in \RR$ and thereby we obtain that $\bigcap_{n=1}^{\infty}\OO(\UU_n,\VV_n)\subset\RR$.
   On the other hand, if $f\in\RR$, for any $n\in\N$ and $(x,y)\in\UU_n\times\VV_n\subset\Omega$ there exists $t\in\R$ such that $f(T_tx)\neq f(T_ty)$. This yields that $f\in \OO(\UU_n,\VV_n)$ for each $n\in\N$. Conclude $\RR=\bigcap_{n=1}^{\infty}\OO(\UU_n,\VV_n)$, i.e. $\RR$ is a dense $G_{\delta}$ subset of $\CC_{A,h}$.
\end{proof}

\subsection{Proof of the theorem}
\begin{proof}[Proof of Theorem \ref{theorem: embedding into C(G,R)}]
    Assume $\dim G=n$. Let $\mathfrak g=T_eG$ be the Lie algebra of $G$. Choose a basis $v_1,\ldots,v_n$ of $\mathfrak g$. We define one-dimensional topological flows $\phi^i:\R\times X\to X, 1\leq i\leq n$ by
    \[\phi_t^i(x)=\exp(tv_i)x.\]
    
        We first claim that $\fix(X,G)=\bigcap_{i=1}^n\fix(X,\phi^i)$.
        If $x\in \fix(X,G)$, $gx=x$ for all $g\in G$. In particular, for all $1\leq i\leq n$, $\exp(tv_i)\in G$ and therefore $\phi_t^i(x)=x$ for all $t\in\R$. This yields that $x\in \fix(X,\phi^i)$ for all $1\leq i\leq n$.
        It suffices to prove that $\bigcap_{i=1}^n\fix(X,\phi^i)\subset \fix(X,G)$. Fix $x\in \bigcap_{i=1}^n\fix(X,\phi^i)$. Denote by $G_x=\{g\in G: gx=x\}$ the stabilizer of $x$. Then $G_x$ is a closed subgroup of $G$ as $G$ acts continuously on $X$. By Theorem \ref{theorem: closed subgroup theorem}, $G_x$ is an embedded Lie subgroup. Since $x\in \fix(X,\phi^i)$, $\exp(tv_i)\in G_x$ for all $t\in\R$ and thus $v_i\in \Lie(G_x):=T_eG_x$. As  $v_1,\ldots,v_n$ are a basis of $\mathfrak g$, it follows $\Lie(G_x)=\mathfrak g$ and $\dim G_x=\dim G=n$. We now consider the inclusion map $\iota: G_x\hookrightarrow G$. Since $G_x$ is an embedded submanifold of $G$,  the differential $d\iota_e: \Lie(G_x)\to \mathfrak g$ is an isomorphism. By the inverse function theorem \cite[Theorem 7.6]{Lee03}, $\iota$ is a local diffeomorphism near $e$. Hence there exists an open neighborhood $U$ of $e$ such that $U\subset G_x$. Conclude $G_x$ is also an open subgroup. Indeed, for any $h\in G_x$, $hU\subset G_x$ is an open neighborhood of $h$. Since $G$ is connected, we obtain that $G_x=G$, which implies that $gx=x$ for all $g\in G$. Conclude $x\in \fix(X,G)$, thereby completing the proof of the claim.

        Note $\fix(X,\phi^i)$ is closed for each $1\leq i\leq n$. 
        Let $h:\fix(X,G)\to\R$ be a topological embedding. 
        We apply Lemma \ref{lemma:one-dimensional flow} to $\phi^i:\R\times X\to X$ and $A:=\fix(X,G)\subset \fix(X,\phi^i)$, and obtain a $G_\delta$ set $\RR_i\subset \CC_{A,h}$ for each $i$ \footnote{Now the reader may see why, in Lemma \ref{lemma:one-dimensional flow}, we only require the values of $h$ to be prescribed on a subset of the fixed point set. This set is exactly the intersection of the fixed point sets of one-dimensional flows constructed from the original $G$-action.}. By the Baire category theorem, the intersection $\bigcap_{i=1}^n\RR_i$ is dense and $G_\delta$ in $\CC_{A,h}$. In particular, it is nonempty. Fix $f\in \bigcap_{i=1}^n\RR_i$. We show that  $J_f:X\to C(G)$ is a $G$-equivariant embedding. Let $J_f(x)=J_f(y)$ with $x\neq y$. Fix $1\leq i\leq n$. Then for all $t\in\R$ it holds
        \begin{equation*}
            f(\phi_t^ix)=f(\exp(tv_i)x)=J_f(x)(\exp(tv_i))=J_f(y)(\exp(tv_i))=f(\exp(tv_i)y)=f(\phi_t^iy).
        \end{equation*}
        Conclude $x,y\in \fix(X,\phi^i)$ since $f\in \RR_i$. Then $x,y\in \bigcap_{i=1}^n\fix(X,\phi^i)=\fix(X,G)$. Since $f\vert _{\fix(X,G)}=h$, we have
        \begin{equation*}
            h(x)=f(x)=J_f(x)(e)=J_f(y)(e)=f(y)=h(y),
        \end{equation*}
        which leads to a contradiction as $h$ is injective. Up to now, we have proven $J_f: X\to C(G)$ is injective. Obviously, $J_f$ is $G$-equivariant and continuous with respect to the compact-open topology. Thus $J_f$ is the desired equivariant embedding.
\end{proof}

\section{A technical lemma}\label{Subsection:a technical lemma}

\begin{lemma}\label{lemma: improvement of regularity}
Let $G$ be a connected Lie group and let $\rho$ be a right-invariant Riemannian
distance.  Then there exists a linear operator
\[
 Q:C_b(G)\longrightarrow C_b^1(G)
\]
satisfying the following properties.
\begin{enumerate}
\item $Q$ is positive (i.e. $u\geq 0$ implies $Qu\geq 0$) and $Q(\11)=1$, where $\11\in C_b(G)$ is the constant function.
\item For every $a\in G$, it holds $R_a \circ Q=Q \circ R_a$.
\item There exists a Borel probability measure $\nu$ on $G$ such that for all $u\in C_b(G)$ and $g\in G$
\begin{equation*}
 Q(u)(g)=\int_Gu(hg)\,d\nu(h).
\end{equation*}
\item $Q$ is injective on $C_b(G)$.
\item There exists $L>0$ such that $\lip_\rho(Qu)\leq L\cdot\osc(u)$ for all $u\in C_b(G)$.
\end{enumerate}
\end{lemma}

\begin{proof}
    Let $\mu$ be the Riemannian volume  induced by a right-invariant metric tensor. 
    Let $\Delta\vert_{C_c^{\infty}(G)}$ be the Laplace--Beltrami  operator defined on $C_c^{\infty}(G)$. Since $G$ is geodesically complete, by \cite[Theorem 2.2]{Gri06}, $\Delta\vert_{C_c^{\infty}(G)}$ is essentially self-adjoint. Denote by $\Delta$ the self-adjoint extension of $\Delta\vert_{C_c^{\infty}(G)}$, whose domain $\DD(\Delta)$ is dense in $L^2(G,\mu)$. Let $\{P_t\}_{t\geq 0}$ be the heat semigroup of $(G,\mu)$ defined on $L^2(G,\mu)$, i.e. $P_t=e^{-t\Delta}$. By \cite[Theorem 3.3]{Gri06}, $(G,\mu)$ has a unique heat kernel $p_t(g,h)$, which is $C^{\infty}$ on $\R_+\times G\times G$. Moreover, the heat kernel satisfies the following properties:
    \begin{enumerate}[label=(\alph*)]
        \item $P_tu(g)=\int_G p_t(g,h)u(h)d\mu(h)$ for all $u\in L^2(G,\mu), t>0$ and $g\in G$;
        \item $p_t(g,h)=p_t(h,g)$ for all $t>0$ and $g,h\in G$;
        \item $p_t(g,h)>0$ and $\int_G p_t(g,h)d\mu(h)\leq 1$ for all $t>0$.
    \end{enumerate}

   For $u\in C_b(G)$ and $t>0$, we define
    \begin{equation*}
        P_tu(g)=\int_G p_t(g,h)u(h)d\mu(h),\quad g\in G, 
    \end{equation*}
    which coincides with the $L^2$-heat semigroup on $C_b(G)\cap L^2(G,\mu)$. This is well-defined in view of property (c). Moreover, for every $u\in C_b(G)$ and $t>0$ we have $P_tu\in C_b(G)$. We will prove in
part~(4) below that $P_tu$ is continuous, and hence that
$P_tu\in C_b(G)$.
    
    We now prove that $P_t\mathbf{1}=1$, i.e. for all $t>0$ and $g\in G$:
    \begin{equation}\label{eq:stochastically complete}
        \int_G p_t(g,h)d\mu(h)=1.
    \end{equation}
    Set $B_R=B_{\rho}(e,R)$ and $V(R)=\mu(B_R)$ for $R>0$. Let $G=\bigcup_{a\in G}B_1a$ be an open cover. Note $\overline{B}_2$ is compact by Theorem \ref{theorem:Hopf-Rinow}. Then there exists $N\in\N$ and $a_1,\ldots,a_N\in G$ such that $\overline{B}_2\subset\bigcup_{i=1}^N B_1a_i$.
    \begin{claim}\label{induction}
        For $m\in\N$, $B_m$ is covered by at most $N^{m-1}$ right argument shifts of $B_1$.
    \end{claim}
    \begin{proof}
        If $m=1$, obviously the claim holds.
        Assume that the claim holds for $1\leq m\in\N$. Fix $x\in B_{m+1}$.
        \begin{enumerate}[label=(\roman*)]
            \item If $x\in B_m$, choose $y=x$. 
            \item Otherwise, set $d_x:=\rho(e,x)\in[m,m+1)$ and choose $s\in(d_x-m,1)$. Let $\gamma:[0,d_x]\to G$ be the unit-speed minimizing geodesic from $e$ to $x$. Set $y=\gamma(d_x-s)$. Then $\rho(e,y)=d_x-s<m$ and $\rho(y,x)=s<1$.
        \end{enumerate}
        Combining (i) and (ii), we conclude that $y\in B_m$ and $\rho(y,x)<1$. Since $\rho$ is right-invariant, $\rho(e,xy^{-1})=\rho(y,x)<1$ and $xy^{-1}\in B_1$. By the induction hypothesis, $y=hw$ for some $h\in B_1$ and the choice of $w$ is at most $N^{m-1}$. Moreover,
        \begin{equation*}
            \rho(e,xy^{-1}h)\leq \rho(e,h)+\rho(h,xy^{-1}h)=\rho(e,h)+\rho(e,xy^{-1})<2.
        \end{equation*}
        This yields that $xy^{-1}h\in B_2$ and therefore $xy^{-1}h=ka_i$ for some $k\in B_1$ and $1\leq i\leq N$. Then $x=xy^{-1}y=xy^{-1}hw=ka_iw$. Note the choice of $a_iw$ is at most $N^{m-1}\cdot N=N^m$. Conclude $B_{m+1}$ is covered by at most $N^m$ right argument shifts of $B_1$ since $x\in B_{m+1}$ is arbitrary.
    \end{proof}
    Recall that the Riemannian volume $\mu$ is right-invariant. By Claim \ref{induction}, for $m\in\N$ it holds
    \begin{equation*}
        V(m)=\mu(B_m)\leq N^{m-1}\mu(B_1)=N^{m-1}V(1).
    \end{equation*}
    Fix $R\geq 1$ and set $m=\lceil R\rceil$. Direct calculation shows that 
    \begin{equation}\label{eq:volume estimate}
        V(R)\leq V(1)e^{CR}.
    \end{equation} 
    where $C=\max\{\log N,1\}$.

        If $G$ is compact, $\mathbf{1}\in C_c^{\infty}(G)$ and $P_t\mathbf{1}=1$ in $L^2(G,\mu)$. 
        Recall that $\R_+\times G\times G\ni (t,g,h)\mapsto p_t(g,h)$ is smooth. Then $P_t\mathbf{1}\equiv 1$.
        If $G$ is non-compact, $\mu(G)=\infty$. Indeed, $\mu$ is a Haar measure and a non-compact locally compact group has infinite Haar measure. Since $B_R$ increases to $G$ as $R\to\infty$, $V(R)=\mu(B_R)\to\infty$ as $R\to\infty$. Choose $R_0\geq 1$ large enough such that $\log V(R)>0$ if $R\geq R_0$. By \eqref{eq:volume estimate} we obtain that $0<\log V(R)\leq \log V(1)+CR$ if $R\geq R_0$. Thus
        \begin{equation*}
            \int_{R_0}^\infty\frac{R\,dR}{\log V(R)}\geq\int_{R_0}^\infty\frac{R\,dR}{\log V(1)+CR}=\infty.
        \end{equation*}
        By \cite[Theorem 3.13]{Gri06}, $(G,\mu)$ is stochastically complete, i.e. for all $t>0$ and $g\in G$:
        \begin{equation*}
          \int_G p_t(h,g)d\mu(h)=1.
        \end{equation*}
        Since $p_t(g,h)=p_t(h,g)$ for all $t>0$ and $g,h\in G$, we have $P_t\mathbf{1}=1$ for all $t>0$.
    Up to now, we have proven \eqref{eq:stochastically complete} whenever $G$ is compact or non-compact.

    Fix $\lambda>0$. We define $Q$ on $C_b(G)$ by
    \begin{equation*}
        Qu(g)=\lambda\int_0^{\infty}e^{-\lambda t}P_tu(g)dt.
    \end{equation*}
    It is clear that $Q$ is linear. We now verify that $Q$ satisfies the desired five properties.
    \begin{enumerate}[label=(\arabic*)]
        \item Note that $Qu\geq 0$ if $u\geq 0$ and 
        \begin{equation*}
            Q\mathbf{1}=\lambda\int_0^{\infty}e^{-\lambda t}P_t\mathbf{1}dt=\lambda\int_0^{\infty}e^{-\lambda t}dt=1.
        \end{equation*}
        \item  We proceed to the proof of the second property of $Q$ in the statement. Fix $a\in G$. Recall that $R_au(g)=u(ga)$ for $g\in G$ and $u\in C_b(G)$. Note the Laplace--Beltrami  operator $\Delta$ is invariant under any isometry\footnote{Interested readers can find a proof of this fact in the note of Gonzalez \cite[pp.105--106]{Gon10}.} on $G$. Then 
        \begin{equation}\label{eq:Laplace commutes with right shift}
            \Delta(R_au)=R_a(\Delta u)
        \end{equation}
        for any $u\in C_c^{\infty}(G)$. Since $\mu$ is right-invariant, $R_a$ is a bounded operator on $L^2(G,\mu)$ because 
        \begin{equation*}
           \langle R_au,R_av\rangle_{L^2(G,\mu)} :=\int_G u(ga)v(ga)d\mu(g)=\int_G u(g)v(g)d\mu(g)=\langle u,v\rangle_{L^2(G,\mu)}.
        \end{equation*} 
        Then \eqref{eq:Laplace commutes with right shift} extends to all $u\in\DD(\Delta)$ as $C_c^{\infty}(G)$ is a core for the essentially self-adjoint operator $\Delta$. 
        Consequently, $R_a$ commutes with the resolvent\footnote{Note $(\Delta-i\Id)^{-1}$ is well-defined since the spectrum of a self-adjoint operator is a subset of $\R$ \cite[Corollary 3.14]{Sch12}.} $(\Delta-i\Id)^{-1}$ on $L^2(G,\mu)$
        because $R_a(\Delta-i\Id)=(\Delta-i\Id)R_a$ on $\DD(\Delta)$ and $R_a$ is bounded.
        By Stone's formulas \cite[Proposition 5.14]{Sch12}, we conclude that $R_a$ commutes with all spectral projections of $\Delta$. In particular, for every bounded Borel function $f:\R\to\mathbb{C}$, by the spectral calculus it holds $R_af(\Delta)=f(\Delta)R_a$. Taking $f(\lambda)=e^{-t\lambda}$ for fixed $t>0$, we further have $R_a\circ P_t=P_t\circ R_a$ on $L^2(G,\mu)$. For $t>0,u\in C_c^{\infty}(G)$ and $g\in G$, it follows
        \begin{equation*}
        \begin{split}
            &\int_Gp_t(g,ha^{-1})u(h)d\mu(h)=\int_Gp_t(g,h)u(ha)d\mu(h)\\
            &=P_t(R_au)(g)=R_a(P_tu)(g)=P_t(u)(ga)\\
            &=\int_Gp_t(ga,h)u(h)d\mu(h)
        \end{split}
        \end{equation*}
        as $\mu$ is right-invariant. 
        Since the heat kernel $p_t(g,h)$ is unique and smooth, for $t>0$ and $g,h\in G$        
        \begin{equation}\label{eq:heat kernel uniqueness}
            p_t(g,ha^{-1})=p_t(ga,h).
        \end{equation}
        Combining \eqref{eq:heat kernel uniqueness} and the right invariance of $\mu$, for any $u\in C_b(G)$ and $g\in G$ it holds
        \begin{equation*}
            \begin{split}
                P_t(R_au)(g)&=\int_Gp_t(g,h)u(ha)d\mu(h)=\int_Gp_t(g,ka^{-1})u(k)d\mu(k)\\
                &=\int_Gp_t(ga,k)u(k)d\mu(k)=P_t(u)(ga)=R_a(P_tu)(g).
            \end{split}
        \end{equation*}
        Using the definition of $Q$, we conclude that $Q\circ R_a(u)=R_a\circ Q(u)$ for every $u\in C_b(G)$.
        \item We now define a Borel probability measure on $G$ by
        \begin{equation*}
            \nu(E)=\lambda\int_0^\infty e^{-\lambda t}\left(\int_Ep_t(e,h)\,d\mu(h)\right)dt,
        \end{equation*}
        where $E\subset G$ is Borel. $\nu$ is a probability measure as $\nu(G)=\lambda\int_0^{\infty}e^{-\lambda t}dt=1$ using \eqref{eq:stochastically complete}. Moreover, we have $Q(u)(e)=\int_Gu(h)d\nu(h)$ for all $u\in C_b(G)$. Indeed, by Tonelli's theorem, for any Borel $E\subset G$
        \begin{equation*}
            \begin{split}
                \nu(E)&=\lambda\int_0^\infty e^{-\lambda t}\left(\int_Gp_t(e,h)\mathbf{1}_E(h)\,d\mu(h)\right)dt\\
                &=\int_G\mathbf{1}_E(h)\left(\lambda \int_0^\infty e^{-\lambda t}p_t(e,h)dt\right)d\mu(h)
            \end{split}
        \end{equation*}
        and therefore
        \begin{equation*}
            \begin{split}
                \int_Gu(h)d\nu(h)&=\int_Gu(h)\left(\lambda \int_0^\infty e^{-\lambda t}p_t(e,h)dt\right)d\mu(h)\\
                &=\lambda\int_0^\infty e^{-\lambda t}\left(\int_Gp_t(e,h)u(h)\,d\mu(h)\right)dt\\
                &=\lambda\int_0^\infty e^{-\lambda t}P_tu(e)dt=Q(u)(e).
            \end{split}
        \end{equation*}
        Using the second property of $Q$, for any $g\in G$
        \begin{equation*}
            \begin{split}
                (Qu)(g)&=R_g\circ Q(u)(e)=Q\circ R_g(u)(e)\\
                &=\int_G(R_gu)(h)d\nu(h)=\int_G u(hg)d\nu(h).
            \end{split}
        \end{equation*}
        This verifies the third property of $Q$.
        \item We now prove that $Qu\in C_b(G)$ for $u\in C_b(G)$. Choose two sequences $(t_j)_{j\in\N}$ and $(g_j)_{j\in\N}$ such that $t_j\to t$ and $\rho(g_j,g)\to 0$ for some $t>0$ and $g\in G$. Fix $h\in G$. Since $p_t(g,h)$ is $C^{\infty}$, $p_{t_j}(g_j,h)\to p_t(g,h)$. Note $p_{t_j}(g_j,h)\geq 0, p_t(g,h)\geq 0$ and by \eqref{eq:stochastically complete} we have
        \begin{equation*}
            \int_G p_{t_j}(g_j,h)d\mu(h)=\int_G p_t(g,h)d\mu(h)=1.
        \end{equation*}
         By Lemma \ref{lemma:Scheffe}, $p_{t_j}(g_j,\cdot)\to p_t(g,\cdot)$ in $L^1(G,\mu)$. Moreover,
        \begin{equation*}
            |P_{t_j}u(g_j)-P_tu(g)|\leq \|u\|_{\infty}\cdot \|p_{t_j}(g_j,\cdot)-p_t(g,\cdot)\|_{L^1(G,\mu)}\to0
        \end{equation*}
        as $j\to\infty$. Thus $(t,g)\mapsto P_tu(g)$ is continuous on $\R_+\times G$. Fix $0<\varepsilon<M$. Set
        \begin{equation*}
            Q_{\varepsilon,M}u(g)=\lambda\int_{\varepsilon}^M e^{-\lambda t}P_tu(g)\,dt
        \end{equation*}
        for $u\in C_b(G)$ and $g\in G$.
        By the dominated convergence theorem, $Q_{\varepsilon,M}u$ is continuous. Since $|P_tu(g)|\leq \|u\|_{\infty}$, 
        \begin{equation}\label{eq:Qu_leq_u}
            \|Qu\|_{\infty}\leq \lambda\int_0^{\infty}e^{-\lambda t}dt \cdot\|u\|_{\infty}=\|u\|_{\infty}.
        \end{equation}
     Furthermore,
        \begin{equation*}
            \begin{split}
                \|Qu-Q_{\varepsilon,M}u\|_{\infty}&\leq\left(\lambda\int_0^{\varepsilon}e^{-\lambda t}dt+\lambda\int_{M}^{\infty}e^{-\lambda t}dt\right)\cdot\|u\|_{\infty}\\
                &=\bigl(1-e^{-\lambda\varepsilon}+e^{-\lambda M}\bigr)\|u\|_{\infty}.
            \end{split}
        \end{equation*}
        Letting $\varepsilon\to0$ and $M\to\infty$, we obtain that $Qu\in C_b(G)$ since $Q_{\varepsilon,M}u$ uniformly converges to $Qu$ on $G$.

        We continue to prove that $Q:C_b(G)\to C_b(G)$ is injective.
        \begin{claim}\label{claim:strongly continuity}
            The heat semigroup $\{P_t\}_{t\geq0}$ is strongly continuous in $L^1(G,\mu)$.
        \end{claim}
        \begin{proof}
            Note $(G,\rho)$ is metrically complete by Theorem \ref{theorem:Hopf-Rinow}. This implies that every compactly supported continuous function is uniformly continuous. Let $\varphi\in C_c(G)$ and denote $K=\supp \varphi$. Fix $\varepsilon>0$ and choose $\delta>0$ such that $|\varphi(g)-\varphi(h)|<\varepsilon$ if $\rho(g,h)<\delta$. By \eqref{eq:stochastically complete}
            \begin{equation*}
                   \int_G \bigl(P_t\varphi(g)-\varphi(g)\bigr)d\mu(g)
                   =\int_G\int_Gp_t(g,h)(\varphi(h)-\varphi(g))d\mu(h)d\mu(g).
            \end{equation*}
            Denote
            \begin{equation*}
                \begin{split}
                    &I_t^1:=\int_G\int_{B_\rho(g,\delta)}p_t(g,h)|\varphi(h)-\varphi(g)|d\mu(h)d\mu(g),\\
                    &I_t^2:=\int_G\int_{G\setminus B_\rho(g,\delta)}p_t(g,h)|\varphi(h)-\varphi(g)|d\mu(h)d\mu(g)
                \end{split}
            \end{equation*}
            respectively.
            On the one hand, 
            \begin{equation*}
            \begin{split}
                I_t^1&\leq  \int_G\int_{B_\rho(g,\delta)}p_t(g,h)|\varphi(g)-\varphi(h)|(\mathbf{1}_K(g)+\mathbf{1}_K(h))d\mu(h)d\mu(g) \\ 
                &\leq \varepsilon \int_G\int_G p_t(g,h)(\mathbf{1}_K(g)+\mathbf{1}_K(h))d\mu(h)d\mu(g)=2\varepsilon\mu(K).
            \end{split}
            \end{equation*}
            On the other hand, recall $p_t(g,h)$ is a regular fundamental solution\footnote{For the definition of \textit{regular fundamental solution} of the heat equation, see \cite[pp.102--103]{Gri06}.}. This implies that for any open set $U$ containing $g$, $\int_U p_t(h,g)d\mu(h)\to1$ as $t\to0$. In particular, for fixed $g\in G$
            \begin{equation*}
                \int_{B_\rho(g,\delta)}p_t(g,h)d\mu(h)=\int_{B_\rho(g,\delta)}p_t(h,g)d\mu(h)\to1
            \end{equation*}
            as $t\to0$. Thus $I_t^2\to 0$ as $t\to0$ by the dominated convergence theorem. Letting $\varepsilon\to0$, we conclude that for every $\varphi\in C_c(G)$ it follows $\|P_t\varphi-\varphi\|_{L^1(G,\mu)}\to0$ as $t\to0$. Since $C_c(G)$ is dense in $L^1(G,\mu)$, we further have $P_t\psi\to\psi$ in $L^1(G,\mu)$ for any $\psi\in L^1(G,\mu)$.
        \end{proof}
        \begin{claim}\label{claim:absolute continuity}
            Fix $\psi\in C_c^{\infty}(G)$. The map $t\mapsto P_t \psi$ is absolutely continuous in $L^1(G,\mu)$.
        \end{claim}
        \begin{proof}
            Note $\Delta\psi\in C_c^{\infty}(G)\subset L^1(G,\mu)\cap L^2(G,\mu)$.
         Moreover, 
               \begin{equation}\label{eq:L2 Bochner integral}
                  P_t\psi-\psi=-\int_0^t P_s\Delta\psi ds
            \end{equation}
            holds in $L^2(G,\mu)$, where the RHS is an $L^2$-Bochner integral.
            We first show that \eqref{eq:L2 Bochner integral} holds in $L^1(G,\mu)$ as well. We divide the interval $[0,t]$ into $n$ intervals of length $\frac{t}{n}$:
            \[0=s_0<s_1<\cdots<s_n=t,\qquad s_k-s_{k-1}=\frac{t}{n} \text{ for }\forall1\leq k\leq n.\]
            Let $R_n:=\sum_{k=1}^n(s_k-s_{k-1})P_{s_k}\Delta\psi$ be the Riemann sum. Since $s\mapsto P_s\Delta\psi$ is strongly continuous in $L^2(G,\mu)$  and $\|P_s\Delta\psi\|_{L^2(G,\mu)}\leq \|\Delta\psi\|_{L^2(G,\mu)}$, $R_n$ converges to $\int_0^tP_s\Delta\psi ds$ in $L^2(G,\mu)$. Combining \eqref{eq:L2 Bochner integral}, we have $R_n$ converges to $-(P_t\psi-\psi)$ in $L^2(G,\mu)$. On the other hand, by Claim \ref{claim:strongly continuity}, $s\mapsto P_s\Delta\psi$ is strongly continuous in $L^1(G,\mu)$ and $\|P_s\Delta\psi\|_{L^1(G,\mu)}\leq \|\Delta\psi\|_{L^1(G,\mu)}$. Then $s\mapsto P_s\Delta\psi$ is Bochner integrable and $R_n$ converges to $\int_0^tP_s\Delta\psi ds$ in $L^1(G,\mu)$. After passing to a subsequence, $R_n$ converges to the same limit on  $L^1(G,\mu)\cap L^2(G,\mu)$, Consequently, \eqref{eq:L2 Bochner integral} holds in $L^1(G,\mu)$. By Bochner's theorem, the map $t\mapsto P_t \psi$ is absolutely continuous in $L^1(G,\mu)$ and for a.e. $t>0$,
            \begin{equation*}
                \left\|\frac{P_{t+s}\psi-P_t\psi}{s}+P_t\Delta\psi\right\|_{L^1(G,\mu)}\to 0
            \end{equation*}
            as $s\to0$.
        \end{proof}
        Fix $\lambda>0$ and $\psi\in C_c^{\infty}(G)$. Define $F_{\psi}(t)=e^{-\lambda t}P_t\psi, t\geq 0$. By Claim \ref{claim:absolute continuity}, the map $t\mapsto F_{\psi}(t)$ is absolutely continuous and 
        \begin{equation*}\label{eq:derivative of F}
            F'_{\psi}(t)=
            -e^{-\lambda t}P_t(\lambda+\Delta)\psi
        \end{equation*}
        for a.e. $t>0$. Obviously, $F'_{\psi}(t)\in L^1(G,\mu)$ since
        \begin{equation*}
            \|F'_{\psi}(t)\|_{L^1(G,\mu)}\leq e^{-\lambda t}(\lambda\|\psi\|_{L^1(G,\mu)}+\|\Delta\psi\|_{L^1(G,\mu)}).
        \end{equation*}
        Moreover, we have $F_{\psi}(t)\to\psi$ in $L^1(G,\mu)$ as $t\to0$ and $\|F_{\psi}(t)\|_{L^1(G,\mu)}\leq e^{-\lambda t}\|\psi\|_{L^1(G,\mu)}\to0$ as $t\to\infty$.

        Since $p_t(g,h)=p_t(h,g)$, by Fubini's theorem, for $u\in L^{\infty}(G,\mu)$ and $\zeta\in L^1(G,\mu)$ it holds $\langle P_tu,\zeta\rangle=\langle u,P_t\zeta\rangle$,
        where $\langle\xi,\zeta\rangle=\int_G \xi(g)\zeta(g)d\mu(g)$ denotes the $L^{\infty}-L^1$ dual product. Using Fubini's theorem again, for $u\in C_b(G)$
        \begin{equation*}
            \begin{split}
                \langle Qu,(\lambda+\Delta)\psi\rangle&=\lambda\int_0^\infty e^{-\lambda t} \langle P_tu,(\lambda+\Delta)\psi\rangle\,dt\\
                     &=\lambda\int_0^\infty e^{-\lambda t} \langle u,P_t(\lambda+\Delta)\psi\rangle\,dt\\
                     &=-\lambda\int_0^\infty\langle u,F_\psi'(t)\rangle\,dt\\
                     &=-\lambda\bigl[\langle u,F_\psi(t)\rangle\bigr]_{t=0}^{t=\infty}\\
                     &=\lambda\langle u,\psi\rangle.
            \end{split}
        \end{equation*}

        Conclude that $\langle Qu,(\lambda+\Delta)\psi\rangle=\lambda\langle
u,\psi\rangle$ for every $\psi\in C_c^\infty(G)$. Since $\Delta$ is
formally self-adjoint, this is precisely the statement that
$(\lambda+\Delta)Qu=\lambda u$ in the distributional sense. In
particular, if $Qu=0$, then $(\lambda+\Delta)Qu=0$, and hence $\lambda
u=0$ as a distribution. Since $\lambda>0$, it follows that $u=0$ in
the distributional sense, and therefore $u(g)=0$ for $\mu$-a.e. $g\in
G$. Since $u\in C_b(G)$ is continuous and $\mu$ has full support,
every nonempty open subset of $G$ has positive $\mu$-measure. Thus $u$
cannot be nonzero at any point of $G$, and consequently $u\equiv0$.
This proves that $Q$ is injective and verifies the fourth property of
$Q$.
        \item Set $v=Qu$. Fix $p>n:=\dim G$. We choose a coordinate chart $\kappa: U_2\to\kappa(U_2)\subset\R^n$ containing $e$ and nested open sets 
        \[e\in V\Subset U_0 \Subset U_1 \Subset U_2\]
        whose images under $\kappa$ are concentric open balls. In particular, $\kappa(U_1)$ is a smooth Euclidean domain. We consider the strictly elliptic operator $\LL:=-(\lambda+\Delta)$.
        In the coordinate chart $\kappa$, $\LL$ becomes an elliptic second-order differential operator on $\kappa(U_2)\subset\R^n$ with smooth coefficients. 
       Since $U_1\Subset U_2$, this operator is uniformly elliptic
on $\kappa(U_1)$.

        Since $u\in C_b(G)$, we have $u\circ\kappa^{-1}\in L^{\infty}(\kappa(U_1))\subset L^p(\kappa(U_1))$.
Let $\LL_\kappa$ denote the coordinate representation of $\LL$. By \cite[Theorem 9.15]{GT01}, there exists a unique $w_1\in W^{2,p}(\kappa(U_1))\cap W_0^{1,p}(\kappa(U_1))$ satisfying $\LL_\kappa w_1=-\lambda u\circ\kappa^{-1}$ in the distributional sense on $\kappa(U_1)$, because we showed above that $(\lambda+\Delta)v=\lambda u$ in the distributional sense. Set $z_1:=v\circ\kappa^{-1}-w_1$. Then $\LL_\kappa z_1=0$ in the distributional sense on $\kappa(U_1)$.

Choose open balls $O,O_1$ with $O\Subset O_1\Subset\kappa(U_1)$ and $\eta\in C_c^\infty(\kappa(U_1))$ such that $\eta\equiv1$ on $O_1$ and $0\leq\eta\leq1$ on $\kappa(U_1)$. Let $\widetilde z_1=\eta z_1$. Then $\widetilde z_1$ has compact support contained in $\kappa(U_1)$ and $\widetilde z_1\equiv z_1$ on $O$. By the elliptic parametrix theorem\footnote{Intuitively speaking, it states that for an elliptic differential operator with smooth coefficients, there exists a properly supported pseudodifferential operator $P$ (called a parametrix) that acts as an \textit{approximate inverse} up to a smoothing error. Moreover, $P$ does not enlarge singular support.} \cite[Theorem 18.1.24]{Hor07}, there exists a properly supported local parametrix $P$ for $\LL_\kappa$ such that $P\LL_\kappa=\Id+R$, where $R$ is a smoothing operator, that is, an operator that sends distributions to smooth functions. Then
\begin{equation*}
    P(\LL_\kappa\widetilde z_1)=\widetilde z_1+R\widetilde z_1.
\end{equation*}
Since $R\widetilde z_1\in C^\infty(\kappa(U_1))$ and $P$ is pseudolocal, taking singular support on both sides gives
\begin{equation*}
    \sing\supp\widetilde z_1=\sing\supp P(\LL_\kappa\widetilde z_1)\subset\sing\supp(\LL_\kappa\widetilde z_1).
\end{equation*}
Writing $M_\eta$ for multiplication by $\eta$, direct calculation shows that
\begin{equation*}
    \LL_\kappa\widetilde z_1=\LL_\kappa(\eta z_1)=\eta\LL_\kappa z_1+[\LL_\kappa,M_\eta]z_1.
\end{equation*}
The commutator $[\LL_\kappa,M_\eta]:=\LL_\kappa\circ M_\eta-M_\eta\circ\LL_\kappa$ is a differential operator of order at most one whose coefficients vanish on $O_1$. Hence $\sing\supp(\LL_\kappa\widetilde z_1)\cap O=\emptyset$ and thus $\sing\supp\widetilde z_1\cap O=\emptyset$. This yields $\widetilde z_1\in C^\infty(O)$. Conclude $z_1\in C^\infty(O)$ since $\widetilde z_1\equiv z_1$ on $O$. Then $z_1\in C^\infty(\kappa(U_1))$ as $O\Subset\kappa(U_1)$ is arbitrary. Since $\kappa(U_0)\Subset\kappa(U_1)$, we have $z_1\in W^{2,p}(\kappa(U_0))$. Therefore $v\circ\kappa^{-1}=w_1+z_1\in W^{2,p}(\kappa(U_0))$.

For arbitrary $g\in G$, choosing a coordinate chart $(U^{(g)},\kappa^{(g)})$ near $g$ and repeating the above argument, we obtain $v\circ(\kappa^{(g)})^{-1}\in W^{2,p}_{\loc}(\kappa^{(g)}(U^{(g)}))$.

        We apply the $L^p$ estimate \cite[Theorem 9.11]{GT01} to  to
$ \LL_\kappa(v\circ\kappa^{-1})
 =-\lambda(u\circ\kappa^{-1})$ on $\kappa(U_0)$. Since $\kappa(V)\Subset\kappa(U_0)$, we obtain
\begin{equation}\label{eq:Lp estimate}
 \|v\circ\kappa^{-1}\|_{W^{2,p}(\kappa(V))}
 \leq C_p\Bigl(
 \lambda\|u\circ\kappa^{-1}\|_{L^p(\kappa(U_0))}
 +\|v\circ\kappa^{-1}\|_{L^p(\kappa(U_0))}
 \Bigr),
\end{equation}
where $C_p>0$ depends on $n,p,\kappa(V),\kappa(U_0)$
and the coefficients of $\LL_\kappa$.
Fix $g\in G$. Let $\rr_g(h)=hg$ and define
\[
    \kappa_g:=\kappa\circ\rr_g^{-1}
    :U_2g\longrightarrow\kappa(U_2).
\]
By \eqref{eq:Laplace commutes with right shift}, the
coordinate expression of $\LL$ in the chart $\kappa_g$
is the same as in the chart $\kappa$. Moreover,
$
    \kappa_g(U_0g)=\kappa(U_0)$ and
$
    \kappa_g(Vg)=\kappa(V).
$
Set
$
    v_g:=v\circ\kappa_g^{-1},
$ and
    $u_g:=u\circ\kappa_g^{-1}.$
Then
$    \LL_\kappa v_g=-\lambda u_g$
    on $\kappa(U_0)$.
Applying the same $L^p$ estimate gives
\begin{equation}\label{eq:Lp estimate independent of g}
\begin{split}
    \|v_g\|_{W^{2,p}(\kappa(V))}
    &\leq C_p\Bigl(
        \lambda\|u_g\|_{L^p(\kappa(U_0))}
        +\|v_g\|_{L^p(\kappa(U_0))}
    \Bigr)\\
    &\leq C_p(\lambda+1)
        |\kappa(U_0)|^{1/p}\|u\|_\infty.
\end{split}
\end{equation}
Here all coordinate $L^p$ norms use Lebesgue measure,
and $|\kappa(U_0)|$ denotes Lebesgue measure.
The constant $C_p$ is independent of $g$, since the
Euclidean domains and the coefficients of the coordinate
operator are independent of $g$. In the last inequality
we used
$
    \|u_g\|_{L^p(\kappa(U_0))}
    \leq |\kappa(U_0)|^{1/p}\|u\|_\infty
$
and the analogous estimate for $v_g$, together with
$\|v\|_\infty=\|Qu\|_\infty\leq\|u\|_\infty$ (see Equation \eqref{eq:Qu_leq_u}).
       
      Since $p>n$, $1<2-\frac{n}{p}<2$. By the Sobolev embedding \cite[Theorem 7.26 (ii)]{GT01}, $W^{2,p}(\kappa(V))\hookrightarrow C^{1,1-\frac{n}{p}}(\overline{\kappa(V)})$. Applying this to $v_g$ and using \eqref{eq:Lp estimate independent of g}, we obtain a constant $C'_{p,\lambda}>0$, independent of $g$, such that $\|Dv_g\|_{L^\infty(\kappa(V))}\leq C'_{p,\lambda}\|u\|_\infty$. Since $\rr_g$ is an isometry and $\kappa_g=\kappa\circ\rr_g^{-1}$, the coordinate matrix of the Riemannian metric in the chart $\kappa_g$ is the same as in the chart $\kappa$. By compactness of $\overline{\kappa(V)}$, it follows that there exists $C_\kappa>0$, independent of $g$, such that $\sup_{y\in Vg}|\nabla v(y)|_{\gg}\leq C_\kappa\|Dv_g\|_{L^\infty(\kappa(V))}$. Therefore there exists $C_{p,\lambda}>0$ such that $\sup_{y\in Vg}|\nabla v(y)|_{\gg}\leq C_{p,\lambda}\|u\|_\infty$ for every $g\in G$. Note $G\subset\bigcup_{g\in G}Vg$ since $e\in V$. Hence
        \begin{equation}\label{eq:global estimate}
            \|\nabla Qu\|_{\infty}\leq C_{p,\lambda} \|u\|_{\infty}.
        \end{equation}
        Denote $c_u=(\sup_G u+\inf_G u)/2$. Replacing $u$ with $u-c_u$ in \eqref{eq:global estimate}, noting $\|u-c_u\|_\infty
=\frac12\osc(u)$, we obtain that 
        \begin{equation*}
            \|\nabla Qu\|_{\infty}\leq\frac{C_{p,\lambda}}{2}\cdot\osc(u).
        \end{equation*}
        Thus $Qu\in C_b^1(G)$. 

        Finally, for $g,h\in G$, let $\gamma_1$ be the unit-speed minimizing geodesic from $g$ to $h$. Then
        \begin{equation*}
            |Qu(g)-Qu(h)|\leq \int_0^{\rho(g,h)}|\nabla Q(u)(\gamma_1(s))|ds \leq \|\nabla Qu\|_{\infty}\cdot\rho(g,h).
        \end{equation*}
        Since $g$ and $h$ are arbitrary, $\lip_{\rho}(Qu)\leq L\cdot\osc(u)$ with $L=C_{p,\lambda}/2>0$.
    \end{enumerate}
\end{proof}

\section{Proof of the main theorem}\label{Section:proof of the main result}

\begin{proof}[Proof of Theorem \ref{theorem: main theorem}]
We will only prove the “if” part as the “only if” part of the statement is trivial.
Assume that $\fix(X,G)$ embeds into $[0,1]$ and also into $\R$.
 By Theorem \ref{theorem: embedding into C(G,R)}, there exists $f_0\in C(X)$ such that $J_{f_0}:X\to C(G)$  \footnote{Recall $J_f: X\to C(G)$ is defined in Subsection \ref{subsection:spaces}.} is an equivariant embedding.

Recall $d$ is a compatible right-invariant metric on $G$. Let $\rho$ be a right-invariant Riemannian distance on $G$. Let $Q:C_b(G)\longrightarrow C_b^1(G)$ and $L>0$ be given
by Lemma \ref{lemma: improvement of regularity}.  Let $c,r,\eta>0$ be given by
Lemma \ref{lemma: comparison between compatible metrics}.  Choose
\begin{equation}\label{eq:delta-choice}
 0<\delta<\min\left\{\frac12,\eta,\frac{c}{L}\right\}
\end{equation}
and a closed interval  $[a_1,b_1]\subset(0,1)$ of length
$b_1-a_1=\delta$.
Denote $m_0=\min_Xf_0$ and $M_0=\max_Xf_0$.  Choose $\alpha>0$ small enough and $\beta\in\R$ such that
$\alpha(M_0-m_0)<\delta$ and
$\alpha[m_0,M_0]+\beta\subset [a_1,b_1]$. Define $f: X\to [a_1,b_1]$ by $f:=\alpha f_0+\beta$.  Then $\osc(J_f(x))\leq\delta$ for all $x\in X$.
Obviously,
$J_f: X\to C_b(G)$ is also a $G$-equivariant embedding.

For $x\in X$, define $q(x)=Q(J_f(x))(e)$. Let $\nu$ be the
probability measure given by Lemma \ref{lemma: improvement of regularity} $(3)$.  Thus
\begin{equation*}
 q(x)=\int_GJ_f(x)(h)d\nu(h)=\int_G f(hx)\,d\nu(h).
\end{equation*}
By the dominated convergence theorem, $q$ is continuous. We now prove that $J_q$ is the desired $G$-equivariant embedding from $X$ to $\lip_1(G,I)$. 

First, we claim that $J_q(x)=Q(J_f(x))$ for every $x\in X$. Indeed, for every $a,g\in G$, note 
\[
 J_f(ax)(g)=f(gax)=J_f(x)(ga)=(R_a\circ J_f(x))(g)
\]
and therefore
\begin{align*}
 J_q(x)(a)=q(ax)=Q(J_f(ax))(e)=Q(R_a\circ J_f(x))(e)=R_a\circ Q(J_f(x))(e)=Q(J_f(x))(a),
\end{align*}
where we have used Lemma \ref{lemma: improvement of regularity} $(2)$ in the fourth equality.

If $J_q(x)=J_q(y)$, then $Q(J_f(x))=Q(J_f(y))$ and thus
$J_f(x)=J_f(y)$ since $Q$ is injective on $C_b(G)$ by Lemma \ref{lemma: improvement of regularity} $(4)$. We further have $x=y$ since $J_f$ is injective.  Conclude $J_q$ is injective.

Finally, since
$J_f(x)(G)\subset [a_1,b_1]$, by Lemma \ref{lemma: improvement of regularity} $(1)$
\[
 Q(J_f(x))-a_1=Q(J_f(x)-a_1)\geq0,
 \qquad
 b_1-Q(J_f(x))=Q(b_1-J_f(x))\geq0.
\]
Thus $J_q(x)(G)=Q(J_f(x))(G)\subset [a_1,b_1]$ and $\osc(J_q(x))\leq b_1-a_1=\delta$.  By
Lemma \ref{lemma: improvement of regularity} $(5)$ we have
\begin{equation}\label{eq:rho-lipschitz-filtered-code}
 \lip_\rho(J_q(x))=\lip_\rho(Q(J_f(x)))
 \leq L\cdot\osc(J_f(x))\le L\delta.
\end{equation}
If $\rho(g,h)<r$, combining \eqref{eq:local-metric-comparison},
\eqref{eq:delta-choice} and \eqref{eq:rho-lipschitz-filtered-code},  we obtain that
\[
 |J_q(x)(g)-J_q(x)(h)|
 \le L\delta\,\rho(g,h)
 \le \frac{L\delta}{c}\,d(g,h)
 \le d(g,h).
\]
If $\rho(g,h)\ge r$, combining \eqref{eq:far-metric-comparison} and
\eqref{eq:delta-choice},  we have
\[
 |J_q(x)(g)-J_q(x)(h)|
 \leq \osc(J_q(x))\leq\delta\le\eta\le d(g,h).
\]
Thus $J_q(x)\in\lip_1(G,I)$ for all $x\in X$. Conclude $J_q:X\to\lip_1(G,I)$ is an equivariant embedding as desired.

\end{proof}

\section*{Acknowledgments}

Q.H. is grateful to Ning Zhou for useful discussions on the proof of Lemma \ref{lemma: improvement of regularity}.
The present work was inspired by helpful discussions with Masaki Tsukamoto, for which we are grateful, as well as by earlier joint work of the first two authors with him. Masaki Tsukamoto graciously declined our invitation to join the present paper as a co-author.
Q.H. was partially supported by National Key Research and Development Program of China 2024YFA1013600; China Postdoctoral Science Foundation 2025M773065; Fundamental Research Funds for the Central Universities WK0010250102.
Y.G. was partially supported by the National Science Centre (Poland) grant 2020/39/B/ST1/02329. K.O. was partially supported by National Natural Science Foundation of China (12371196, 12426201).

\bibliographystyle{alpha}
\bibliography{BebutovKakutani}

\end{document}